\documentclass[a4paper]{amsart}
\usepackage{stmaryrd}
\usepackage{graphicx}
\usepackage{amssymb}
\usepackage{amsmath}
\usepackage{amsthm}
\usepackage{amscd}
\usepackage[all,2cell]{xy}

\UseAllTwocells \SilentMatrices
\newtheorem{thm}{Theorem}[section]

\newtheorem{cor}[thm]{Corollary}
\newtheorem{lem}[thm]{Lemma}
\newtheorem{exm}[thm]{Example}

\newtheorem{fact}[thm]{Fact}
\newtheorem{prop}[thm]{Proposition}
\theoremstyle{definition}

\theoremstyle{remark}
\newtheorem{rem}[thm]{\bf Remark}
\numberwithin{equation}{section}

\begin{document}
\title[A note on Serre duality and equivariantization]
{A note on Serre duality and equivariantization}

\author[Xiao-Wu Chen] {Xiao-Wu Chen}


\date{\today}

\thanks{E-mail:
xwchen$\symbol{64}$mail.ustc.edu.cn}

\maketitle

\dedicatory{}%
\commby{}%

\begin{abstract}
For an abelian category with a Serre duality and a finite group action, we compute explicitly the Serre duality
on the category of equivariant objects. Special cases and examples are discussed. In particular, an abelian category with
a periodic Serre duality and its equivariantization are studied.
\end{abstract}

\section{Introduction}

Let $k$ be a field, and let $\mathcal{A}$ be a $k$-linear abelian category with a Serre duality \cite{RV,LZ}. We assume that there is a $k$-linear action on $\mathcal{A}$ by a finite group $G$; see \cite{RR,De,DGNO}. Then the category $\mathcal{A}^G$ of equivariant objects is abelian and naturally $k$-linear. In case that the order of $G$ is invertible in $k$, it is known that the category $\mathcal{A}^G$ has a Serre duality. However, it seems that an explicit description of the Serre duality on $\mathcal{A}^G$, in particular, the Serre functor on $\mathcal{A}^G$,  is not contained in any literature.

The first goal of this note is to describe the Serre duality on $\mathcal{A}^G$ explicitly. The second goal is to study an abelian category $\mathcal{A}$ with a periodic Serre duality. The latter means that  a certain power of the Serre functor on $\mathcal{A}$ is isomorphic to the identity functor. Examples of such abelian categories are the categories of coherent sheaves on weighted projective lines of tubular type or on elliptic curves. Indeed, we are motivated by these examples. We recall that the category of coherent sheaves on a weighted projective line of tubular type is related to the one on an elliptic curve via two different equivariantizations; see \cite{GL87, Lentalk, CCZ}.

We describe the content of this note.  In Section 2, we recall some notation on Serre duality and the notion of a commutator isomorphism which plays a central role.  The basic properties of commutator isomorphisms are collected in Proposition \ref{prop:1}, which is analogous to the results in \cite{Kel}. We mention that the notion of a perfect Serre duality seems to be of interest; see Subsection 2.2. For an abelian category $\mathcal{A}$ with a periodic Serre duality, we  observe the induced homomorphism from the group of isoclasses of auto-equivalences on $\mathcal{A}$ to the group of invertible elements in the center of $\mathcal{A}$; see Proposition \ref{prop:2}.

In Section 3, we calculate explicitly the Serre duality on $\mathcal{A}^G$ in Theorem \ref{thm:1}. Special cases of Theorem \ref{thm:1} are discussed. In particular, we prove that if $\mathcal{A}$ has a periodic Serre duality, so does $\mathcal{A}^G$; however, the orders of the Serre functors on $\mathcal{A}$ and $\mathcal{A}^G$ may differ; see Propositions \ref{prop:A} and \ref{prop:B}. We close the note with the motivating examples, which illustrate these propositions.

Throughout this note $k$ is a field.  We denote by $D={\rm Hom}_k(-, k)$ the duality functor on finite dimensional $k$-vector spaces.

\section{The Serre duality on an abelian category}

In this section, we recall from \cite{RV,LZ} some notation on Serre duality.  Following \cite{Kel}, we study basic properties of commutator isomorphisms. We mention that the notion of a perfect Serre duality seems to be of interest. We observe the induced homomorphism for an abelian category with a periodic Serre duality.

 Let $\mathcal{A}$ be a $k$-linear abelian category which is skeletally small. We assume that $\mathcal{A}$ is \emph{Hom-finite}, that is, for each pair $X, Y$ of objects the Hom space ${\rm Hom}_\mathcal{A}(X, Y)$ is finite dimensional over $k$.

\subsection{The trace map}

Recall by definition that a \emph{Serre duality} on $\mathcal{A}$ is a bifunctorial $k$-linear isomorphism
\begin{align}\label{equ:S}
\phi_{X, Y}\colon D{\rm Ext}_\mathcal{A}^1(X, Y)\stackrel{\sim}\longrightarrow {\rm Hom}_\mathcal{A}(Y, \mathbb{S}X),
\end{align}
where $\mathbb{S}\colon \mathcal{A}\rightarrow \mathcal{A}$ is a $k$-linear auto-equivalence, called the \emph{Serre functor} of $\mathcal{A}$; it is also called the \emph{Auslander-Reiten translation} of $\mathcal{A}$ in the literature. In this case, we say that $\mathcal{A}$ has the Serre duality (\ref{equ:S}); in particular, we observe that the Ext space ${\rm Ext}_\mathcal{A}^1(X, Y)$ is finite dimensional over $k$. It is well known that $\mathcal{A}$ has a Serre duality if and only if $\mathcal{A}$ is hereditary without nonzero projective or injective objects such that its bounded derived category $\mathbf{D}^b(\mathcal{A})$ is Hom-finite with a Serre duality in the sense of \cite[Definition 3.1]{BK}.

Recall that elements in ${\rm Ext}_\mathcal{A}^1(X, Y)$ are equivalence classes $[\xi]$ of extensions
$\xi\colon 0\rightarrow Y\rightarrow E\rightarrow X\rightarrow 0$.  For a morphism $f\colon X'\rightarrow X$, we denote by $\xi.f$ the pullback of $\xi$ along $f$, whose class in ${\rm Ext}_\mathcal{A}^1(X', Y)$ is denoted by $[\xi].f$; similarly, for a morphism
$g\colon Y\rightarrow Y'$, we denote by $g.\xi$ the pushout of $\xi$ along $g$ and by $g.[\xi]$ its class in ${\rm Ext}_\mathcal{A}^1(X, Y')$.

The following standard fact will be used later.

\begin{lem}\label{lem:McL}
The following two statements hold.
\begin{enumerate}
\item Assume that the following diagram is commutative with exact rows.
\[\xymatrix{
\xi\colon 0\ar[r] & X \ar[d]^-{f}\ar[r] &Y \ar[d] \ar[r] & Z \ar[r] \ar[d]^-{g} & 0\\
\xi'\colon 0\ar[r] & X' \ar[r] &Y' \ar[r] & Z' \ar[r] & 0
}\]
Then we have $f.[\xi]=[\xi'].g$ in ${\rm Ext}^1_\mathcal{A}(Z, X')$.
\item Let $\theta\colon F\rightarrow F'$ be a natural transformation between two auto-equivalences on $\mathcal{A}$, and let $[\xi]\in {\rm Ext}^1_\mathcal{A}(Z, X)$. Then we have $\theta_X.[F(\xi)]=[F'(\xi)].\theta_Z$ in ${\rm Ext}^1_\mathcal{A}(FZ, F'X)$.
\end{enumerate}
\end{lem}

\begin{proof}
(1) is due to  \cite[Chapter III, Proposition 1.8]{McL}, while (2) is a direct application of (1) to the following commutative diagram
\[\xymatrix{
F(\xi)\colon & 0\ar[r] & FX \ar[d]^-{\theta_X}\ar[r] & FY\ar[r] \ar[d]^-{\theta_Y} & FZ\ar[r] \ar[d]^-{\theta_Z} & 0\\
F'(\xi)\colon & 0\ar[r] & F'X\ar[r] & F'Y\ar[r] & F'Z\ar[r] & 0.\\
}\]
\end{proof}

The Serre duality (\ref{equ:S}) yields a non-degenerated pairing
$$\langle -, -\rangle_{X, Y}\colon {\rm Ext}_\mathcal{A}^1(X, Y)\times {\rm Hom}_\mathcal{A}(Y, \mathbb{S}X)\longrightarrow k$$
such that $\langle [\xi], f\rangle_{X,Y}=\phi^{-1}_{X,Y}(f)([\xi])$. For each object $X$, we define the \emph{trace map}
$${\rm Tr}_X\colon {\rm Ext}_\mathcal{A}^1(X, \mathbb{S}X)\longrightarrow k$$
 by ${\rm Tr}_X([\xi])=\langle [\xi], {\rm Id}_{\mathbb{S}X}\rangle_{X, \mathbb{S}X}=\phi^{-1}_{X, \mathbb{S}X}({\rm Id}_{\mathbb{S}X})([\xi])$.

We have the following well-known observations.

\begin{lem}\label{lem:1}
Keep the notation as above. Then we have the following statements:
\begin{enumerate}
\item ${\rm Tr}_X(f.[\xi])=\langle [\xi], f\rangle_{X, Y}$ for any $[\xi]\in {\rm Ext}_\mathcal{A}^1(X,Y)$ and $f\colon Y\rightarrow \mathbb{S}X$;
    \item ${\rm Tr}_{X'}([\xi'].f')={\rm Tr}_X(\mathbb{S}(f').[\xi'])$ for any $[\xi'] \in {\rm Ext}_\mathcal{A}^1(X, \mathbb{S}X')$ and $f'\colon X'\rightarrow X$;
        \item two morphisms $f, f'\colon Y\rightarrow \mathbb{S}X$ are equal if and only if ${\rm Tr}_X(f.[\xi])={\rm Tr}_X(f'.[\xi])$ for each $[\xi]\in {\rm Ext}_\mathcal{A}^1(X, Y)$.
\end{enumerate}
\end{lem}

\begin{proof}
(1) follows from the naturalness of $\phi_{X, Y}$ in the variable $Y$, and (2) follows from (1) and the naturalness of $\phi_{X, Y}$ in the variable  $X$. (3) follows from (1) and the non-degeneratedness of the pairing $\langle -, -\rangle_{X, Y}$.
\end{proof}

\subsection{The commutator isomorphism}

Let $\mathcal{A}$ have the Serre duality (\ref{equ:S}). The following result is implicitly contained in \cite[Subsection 2.3]{Kel}.

\begin{lem}\label{lem:2}
Let $F\colon \mathcal{A}\rightarrow \mathcal{A}$ be a $k$-linear auto-equivalence. Then there is a unique natural isomorphism $\sigma_F\colon F\mathbb{S}\rightarrow \mathbb{S}F$ satisfying
\begin{align}\label{equ:1}
{\rm Tr}_X([\xi])={\rm Tr}_{FX}((\sigma_F)_X.[F(\xi)])
 \end{align}
 for all objects $X$ and $[\xi]\in {\rm Ext}_\mathcal{A}^1(X, \mathbb{S}X)$.
\end{lem}

This unique isomorphism $\sigma_F\colon F\mathbb{S}\rightarrow \mathbb{S}F$ is called the \emph{commutator isomorphism} associated to $F$.

\begin{proof}
The uniqueness of the morphism $(\sigma_F)_X$ follows from Lemma \ref{lem:1}(3).

For the existence of $\sigma_F$, we consider for each pair $X, Y$ of objects in $\mathcal{A}$ the following bifunctorial isomorphisms
\[\xymatrix{
\Theta \colon {\rm Hom}_\mathcal{A}(FY, \mathbb{S}FX) \ar[r]^-{\phi_{FX, FY}^{-1}} & D{\rm Ext}^1_\mathcal{A}(FX, FY)\ar[r]^-{DF} & D{\rm Ext}^1_\mathcal{A}(X, Y)  \ar[ld]_-{{\phi_{X,Y}}}\\
& {\rm Hom}_\mathcal{A}(Y, \mathbb{S}X)\ar[r]^-{F}&  {\rm Hom}_\mathcal{A}(FY, F\mathbb{S}X).}\]
Recall that $F$ is an auto-equivalence on $\mathcal{A}$. We apply Yoneda's lemma to have an isomorphism $\delta_X\colon \mathbb{S}FX\rightarrow F\mathbb{S} X$ such that $\Theta={\rm Hom}_\mathcal{A}(FY, \delta_X)$; this identity implies that for any morphism $f\colon Y\rightarrow \mathbb{S}X$, we have $$\phi^{-1}_{X, Y}(f)(-)=\phi^{-1}_{FX, FY}(\delta_X^{-1}\circ F(f))(F(-))$$ as $k$-linear functions on ${\rm Ext}^1_\mathcal{A}(X, Y)$. Putting $Y=\mathbb{S}X$ and $f={\rm Id}_{\mathbb{S}X}$, we have ${\rm Tr}_X([\xi])={\rm Tr}_{FX}(\delta_X^{-1}.[F(\xi)])$ for each $[\xi]\in {\rm Ext}_\mathcal{A}^1(X, \mathbb{S}X)$.  The isomorphism $\delta_X$ is natural in $X$ by the functorialness of $\Theta$. Set $\sigma_F=\delta^{-1}$. Then we are done.
\end{proof}

We introduce the following notion for later use. The Serre duality (\ref{equ:S}) is said to be \emph{perfect} provided that $\sigma_{\mathbb{S}}={\rm Id}_{\mathbb{S}^2}$, or equivalently, the corresponding trace map satisfies
\begin{align}\label{equ:4}
{\rm Tr}_X([\xi])={\rm Tr}_{\mathbb{S}X}([\mathbb{S}(\xi)])
 \end{align}
 for all objects $X$ and $[\xi]\in {\rm Ext}_\mathcal{A}^1(X, \mathbb{S}X)$. In this case, we say that the category $\mathcal{A}$ has a \emph{perfect Serre duality}. This implies that any Serre duality on $\mathcal{A}$ is perfect; see Remark \ref{rem:1}. The author does not know any example of an abelian category $\mathcal{A}$ with a non-perfect Serre duality; compare the triangulated case in \cite[Lemma 2.1 b)]{Kel}.

\begin{rem}\label{rem:1}
Given the Serre duality (\ref{equ:S}) on $\mathcal{A}$, it is well known that any other Serre duality $\phi'$ on $\mathcal{A}$ differs from $\phi$ by an isomorphism. More precisely, assume that
\begin{align}\label{equ:S'}\phi'_{X, Y}\colon D{\rm Ext}_\mathcal{A}^1(X, Y)\stackrel{\sim}\longrightarrow {\rm Hom}_\mathcal{A}(Y, \mathbb{S}'X)
\end{align}
is another Serre duality. Then there is a unique natural isomorphism $\theta\colon \mathbb{S}'\rightarrow \mathbb{S}$ such that $\phi_{X, Y}={\rm Hom}_\mathcal{A}(Y, \theta_X) \circ \phi'_{X, Y}$. The corresponding trace map is given by ${\rm Tr}'_X([\xi])={\rm Tr}_X(\theta_X.[\xi])$ for each $[\xi]\in {\rm Ext}_\mathcal{A}^1(X, \mathbb{S}'X)$.

We claim that if the Serre duality (\ref{equ:S}) is perfect, so is (\ref{equ:S'}). Indeed, the following identity proves the claim:
\begin{align*}
{\rm Tr}'_{\mathbb{S'}X}([\mathbb{S'}(\xi)])&={\rm Tr}_{\mathbb{S}'X}(\theta_{\mathbb{S'}X}.[\mathbb{S}'(\xi)])\\
&={\rm Tr}_{\mathbb{S}'X}([\mathbb{S}(\xi)].\theta_X)\\
&={\rm Tr}_{\mathbb{S}X}(\mathbb{S}(\theta_{X}).[\mathbb{S}(\xi)]\\
&={\rm Tr}_{\mathbb{S}X}([\mathbb{S}(\theta_X.\xi)])\\
&={\rm Tr}_{X}(\theta_X.[\xi])={\rm Tr}'_X([\xi]).\end{align*}
 Here, the second equality uses the fact $\theta_{\mathbb{S'}X}.[\mathbb{S}'(\xi)]=[\mathbb{S}(\xi)].\theta_X$, which is a consequence of Lemma \ref{lem:McL}(2). The third equality uses Lemma \ref{lem:1}(2), while the fifth uses the perfectness of (\ref{equ:S}).
\end{rem}

\begin{exm}\label{exm:3}
{\rm Let $\mathcal{A}$ have the Serre duality (\ref{equ:S}). We say that the Serre functor $\mathbb{S}$ is \emph{trivial} provided that it is isomorphic to the identity functor. In this case, the Serre duality is perfect. Indeed, by Remark \ref{rem:1} we may assume that $\mathbb{S}$ equals the identity functor. Then the perfectness follows from (\ref{equ:4}).}
\end{exm}

The dependence of the commutator isomorphism $\sigma_F$ on the auto-equivalence $F$ is shown in the following result, where the first statement is due to \cite[Lemma 2.1 b)]{Kel} in a slightly different setup.

\begin{lem}\label{lem:key}
Keep the notation as above. Then the following two statements hold.
\begin{enumerate}
\item For two $k$-linear auto-equivalences $F_1$ and $F_2$ on $\mathcal{A}$, we have $\sigma_F=(\sigma_{F_1}F_2)\circ (F_1\sigma_{F_2})$, where $F=F_1F_2$.
    \item Let $\theta\colon F\rightarrow F'$ be a natural isomorphism between two $k$-linear auto-equivalences $F$ and $F'$ on $\mathcal{A}$. Then we have $\sigma_{F'}\circ \theta\mathbb{S}=\mathbb{S}\theta\circ \sigma_F$.
\end{enumerate}
\end{lem}

\begin{proof}
For (1), it suffices to prove that $(\sigma_F)_X=(\sigma_{F_1})_{F_2X}\circ F_1((\sigma_{F_2})_X)$  for each object $X$. Consider $\Delta={\rm Tr}_{FX}(((\sigma_{F_1})_{F_2X}\circ F_1((\sigma_{F_2})_X)).[F(\xi)])$ for any $[\xi]\in {\rm Ext}_\mathcal{A}^1(X, \mathbb{S}X)$. By Lemma \ref{lem:1}(3), it suffices to show that $\Delta={\rm Tr}_{FX}((\sigma_F)_X.[F(\xi)]).$

Indeed, we observe that $\Delta={\rm Tr}_{F_1(F_2X)}((\sigma_{F_1})_{F_2X}.[F_1( (\sigma_{F_2})_X.F_2(\xi))])$. By (\ref{equ:1})  for both $F_1$ and $F_2$, we have $\Delta={\rm Tr}_{F_2X}((\sigma_{F_2})_X.[F_2(\xi)])={\rm Tr}_X([\xi])$, which equals ${\rm Tr}_{FX}((\sigma_F)_X.[F(\xi)])$  by (\ref{equ:1}) for $F$. Then we are done with (1).

For (2), we take an object $X$ and  $[\xi]\in {\rm Ext}_\mathcal{A}^1(F'X, F\mathbb{S}X)$. Assume that $\theta_{\mathbb{S}X}.[\xi]=[F'(\xi')]$ for a unique $[\xi']\in {\rm Ext}_\mathcal{A}^1(X, \mathbb{S}X)$. We claim that $[F(\xi')]=[\xi].\theta_X$. Indeed, by applying Lemma \ref{lem:McL}(2) to $\theta$ and $[\xi']$, we obtain $\theta_{\mathbb{S}X}.[F(\xi')]=[F'(\xi')].\theta_X$. By $[F'(\xi')]=\theta_{\mathbb{S}X}.[\xi]$,  we have $\theta_{\mathbb{S}X}.[F(\xi')]=\theta_{\mathbb{S}X}.([\xi].\theta_X)$. Since $\theta_{\mathbb{S}X}$ is an isomorphism, we infer the claim.

 We will prove that $(\sigma_{F'})_X\circ \theta_{\mathbb{S}X}=\mathbb{S}(\theta_X)\circ (\sigma_F)_X$. Consider $\nabla={\rm Tr}_{F'X}(((\sigma_{F'})_X\circ \theta_{\mathbb{S}X}).[\xi])$ for $[\xi]\in {\rm Ext}_\mathcal{A}^1(F'X, F\mathbb{S}X)$. By Lemma \ref{lem:1}(3), it suffices to show that $\nabla={\rm Tr}_{F'X}((\mathbb{S}(\theta_X)\circ (\sigma_F)_X).[\xi])$.

Indeed, by (\ref{equ:1}) for $F'$, we have $\nabla={\rm Tr}_{F'X}((\sigma_{F'})_X.[F'(\xi')])={\rm Tr}_X([\xi'])$. Applying (\ref{equ:1}) for $F$, we have $\nabla={\rm Tr}_{FX}((\sigma_F)_X.[F(\xi')])$, which by the claim above equals ${\rm Tr}_{FX}((\sigma_F)_X.([\xi].\theta_X))={\rm Tr}_{FX}(((\sigma_F)_X.[\xi]).\theta_X)$. Applying Lemma \ref{lem:1}(2), we are done.
\end{proof}

We now extend Lemma \ref{lem:key} slightly. Let $F$ be a $k$-linear auto-equivalence on $\mathcal{A}$. For each $d\geq 1$, we define a natural isomorphism $\sigma_F^d\colon F\mathbb{S}^d\rightarrow \mathbb{S}^dF$ inductively as follows: $\sigma_F^1=\sigma_F$ and $\sigma^{d+1}_F=\mathbb{S}\sigma_F^{d}\circ \sigma_F\mathbb{S}^d$ for $d\geq 1$. Here, $\mathbb{S}^d$ denotes the $d$-th power of $\mathbb{S}$. We refer to the isomorphism $\sigma_F^d$ as the \emph{$d$-th commutator isomorphism} associated to $F$.

\begin{prop}\label{prop:1}
Let $\mathcal{A}$ have the Serre duality (\ref{equ:S}) and let $d\geq 1$. We keep the notation as above. Then the following two statements hold.
\begin{enumerate}
\item For two $k$-linear auto-equivalences $F_1$ and $F_2$ on $\mathcal{A}$, we have $\sigma^d_F=(\sigma^d_{F_1}F_2)\circ (F_1\sigma^d_{F_2})$, where $F=F_1F_2$.
    \item Let $\theta\colon F\rightarrow F'$ be a natural isomorphism between two $k$-linear auto-equivalences $F$ and $F'$ on $\mathcal{A}$. Then we have $\sigma^d_{F'}\circ \theta\mathbb{S}^d=\mathbb{S}^d\theta\circ \sigma^d_F$.
\end{enumerate}
\end{prop}

\begin{proof}
The case that $d=1$ is due to Lemma \ref{lem:key}.

We assume by induction that (1) holds for $d-1$. We have the following identity
\begin{align*}
(\sigma^d_{F_1}F_2)\circ (F_1\sigma^d_{F_2})&=(\mathbb{S}\sigma_{F_1}^{d-1}F_2)\circ (\sigma_{F_1}\mathbb{S}^{d-1}F_2)\circ (F_1\mathbb{S}\sigma_{F_2}^{d-1})\circ (F_1\sigma_{F_2}\mathbb{S}^{d-1})\\
&= (\mathbb{S}\sigma_{F_1}^{d-1}F_2)\circ (\mathbb{S}F_1\sigma_{F_2}^{d-1})\circ (\sigma_{F_1}F_2\mathbb{S}^{d-1})\circ (F_1\sigma_{F_2}\mathbb{S}^{d-1})\\
&=\mathbb{S}\sigma^{d-1}_F\circ \sigma_F\mathbb{S}^{d-1}\\
&=\sigma^d_F,
\end{align*}
where the second equality uses the naturalness of $\sigma_{F_1}$ and the third uses the induction hypothesis and Lemma \ref{lem:key}(1). Then we are done with (1). By a similar argument, we prove (2).
\end{proof}

\subsection{The induced homomorphism}

We denote by $Z(\mathcal{A})$ the \emph{center} of $\mathcal{A}$, which is by definition the set consisting of all natural transformations $\lambda\colon {\rm Id}_\mathcal{A}\rightarrow {\rm Id}_\mathcal{A}$. It is a commutative $k$-algebra with multiplication given by the composition of natural transformations.

 For a morphism $f\colon X\rightarrow Y$ and $\lambda\in Z(\mathcal{A})$, we have $f\circ \lambda_X=\lambda_Y\circ f$, both of which will be denoted by $\lambda.f$. In this manner, the Hom space ${\rm Hom}_\mathcal{A}(X, Y)$ is naturally a $Z(\mathcal{A})$-module; moreover, the composition of morphisms is $Z(\mathcal{A})$-bilinear. In other words,  the category $\mathcal{A}$ is naturally $Z(\mathcal{A})$-linear.  We observe that for an endofunctor $F\colon \mathcal{A}\rightarrow \mathcal{A}$, it is $Z(\mathcal{A})$-linear if and only if it is $k$-linear satisfying that $F(\lambda_X)=\lambda_{FX}$ for each object $X$ in $\mathcal{A}$ and $\lambda\in Z(\mathcal{A})$, or equivalently, $F\lambda=\lambda F$ for each $\lambda\in Z(\mathcal{A})$.

The following observation is immediate.

\begin{lem}\label{lem:Z}
 Let $F\colon \mathcal{A}\rightarrow \mathcal{A}$ be an auto-equivalence. Then any natural transformation $F\rightarrow F$ is of the form $\lambda F=F\lambda'$ for  unique $\lambda, \lambda'\in Z(\mathcal{A})$. In case that $F$ is $Z(\mathcal{A})$-linear, we have $\lambda=\lambda'$. \hfill $\square$
\end{lem}

We denote by $Z(\mathcal{A})^\times$ the group consisting of invertible elements in $Z(\mathcal{A})$, that is, natural automorphisms of ${\rm Id}_\mathcal{A}$.

The following result is an immediate consequence of Lemma \ref{lem:key}(2).

\begin{cor}\label{cor:S}
Let $\mathcal{A}$ have the Serre duality (\ref{equ:S}), and let $\lambda\in Z(\mathcal{A})^\times$. Then $\lambda\mathbb{S}=\mathbb{S}\lambda$.\hfill $\square$
\end{cor}

\begin{rem}\label{rem:S} By Corollary \ref{cor:S} it seems that the Serre functor $\mathbb{S}$ on $\mathcal{A}$ is $Z(\mathcal{A})$-linear. However, the author does not know a proof.
\end{rem}

We denote by ${\rm Aut}_k(\mathcal{A})$ the set consisting of isomorphism classes $[F]$ of $k$-linear auto-equivalences $F$ on $\mathcal{A}$, which is a group by the composition of functors. We denote by ${\rm Aut}(\mathcal{A})$ its subgroup formed by $Z(\mathcal{A})$-linear auto-equivalences.

A $k$-linear auto-equivalence $F\colon \mathcal{A}\rightarrow \mathcal{A}$ is \emph{periodic} if there exists a natural isomorphism $\eta\colon F^d\rightarrow {\rm Id}_\mathcal{A}$ for some $d\geq 1$, or equivalently, the element $[F]$ in ${\rm Aut}_k(\mathcal{A})$ has finite order, in which case $\eta$ is called a \emph{periodicity isomorphism} of $F$ with order $d$. The periodicity isomorphism $\eta$ is \emph{compatible} provided that $\eta F=F\eta$. The order of $[F]$ in  ${\rm Aut}_k(\mathcal{A})$  will be called the \emph{order} of $F$.

The category $\mathcal{A}$ is said to have a \emph{periodic Serre duality} provided that  it has a Serre duality with the Serre functor $\mathbb{S}$ periodic. In this case, we fix a  periodicity isomorphism of order $d$
\begin{align}
\eta\colon \mathbb{S}^d \longrightarrow {\rm Id}_\mathcal{A}.
\end{align}
We observe that the order of $\mathbb{S}$ divides $d$. We mention that if $\mathcal{A}$ has periodic Serre duality, then $\mathbf{D}^b(\mathcal{A})$ is fractionally Calabi-Yau.

 Let $F\colon \mathcal{A}\rightarrow \mathcal{A}$ be a $Z(\mathcal{A})$-linear auto-equivalence. In particular, it is $k$-linear. Consider the following natural isomorphisms
 \begin{align}\label{equ:3}
 t_F\colon F\stackrel{F\eta^{-1}}\longrightarrow F\mathbb{S}^d\stackrel{\sigma_F^d}\longrightarrow \mathbb{S}^dF\stackrel{\eta F} \longrightarrow F.
 \end{align}
 By Lemma \ref{lem:Z} there exists a unique $\kappa(F)\in Z(\mathcal{A})^\times$ such that $t_F=\kappa(F) F$. We claim that $\kappa(F)=\kappa(F')$ for any given natural isomorphism $\theta\colon F\rightarrow F'$. Then this gives rise to a well-defined map
\begin{align}\label{equ:lambda}
\kappa\colon {\rm Aut}(\mathcal{A})\longrightarrow Z(\mathcal{A})^\times, \quad [F]\mapsto \kappa(F).
\end{align}

 For the claim, we apply Proposition \ref{prop:1}(2) to infer that $\theta\circ t_F=t_{F'}\circ \theta$; here, we also use the naturalness of $\theta$ and $\eta$. Then we have $\theta\circ (\kappa(F)F)=(\kappa(F')F')\circ \theta$. Since $\kappa(F)$ lies in $Z(\mathcal{A})$, we have $\theta\circ (\kappa(F)F)=(\kappa(F)F')\circ \theta$, and thus $(\kappa(F)F')\circ \theta=(\kappa(F')F')\circ \theta$. Recall that $\theta$ is an isomorphism and that $F'$ is an auto-equivalence. It follows that $\kappa(F)=\kappa(F')$.

 \begin{prop}\label{prop:2}
 Let $\mathcal{A}$ have a periodic Serre duality. We keep the notation as above. Then the map $\kappa$ in (\ref{equ:lambda}) is a group homomorphism.
 \end{prop}

 We will refer to the group homomorphism $\kappa\colon {\rm Aut}(\mathcal{A})\rightarrow Z(\mathcal{A})^\times$ as the \emph{induced homomorphism} of the periodicity isomorphism $\eta$.

 We observe that the periodicity isomorphism $\eta$ is unique up to an element in $Z(\mathcal{A})^\times$. More precisely, for another periodicity isomorphism $\eta'\colon \mathbb{S}^d\rightarrow {\rm Id}_\mathcal{A}$, there is a unique $\lambda\in Z(\mathcal{A})^\times$ with $\eta=\lambda \circ \eta'$.  It follows from the $Z(\mathcal{A})$-linearity of $F$ and Corollary \ref{cor:S} that the induced homomorphism  of $\eta'$ coincides with $\kappa$. In other words, the induced homomorphism is independent of the choice of the periodicity isomorphism.

 \begin{proof}
 Let $F_1$ and $F_2$ be two $Z(\mathcal{A})$-linear auto-equivalences on $\mathcal{A}$. Set $F=F_1F_2$. For the result, it suffices to claim that $\kappa(F)=\kappa(F_1)\kappa(F_2)$. By Proposition \ref{prop:1}(1) we have the first equality of the following identity
 \begin{align*}t_F &=(t_{F_1}F_2)\circ (F_1t_{F_2})\\
                  &= (\kappa(F_1)F_1F_2)\circ (F_1\kappa(F_2)F_2)\\
                  &=(\kappa(F_1)F)\circ (\kappa(F_2)F)\\
                  &=(\kappa(F_1)\kappa(F_2))F,
                  \end{align*}
 where the third equality uses the $Z(\mathcal{A})$-linearity of $F_1$. Recall that $t_F=\kappa(F)F$ and that $F$ is an auto-equivalence. By the uniqueness statement in Lemma \ref{lem:Z} we infer the claim.
 \end{proof}

The following result concerns the compatibility of the periodicity isomorphism $\eta\colon \mathbb{S}^d\rightarrow {\rm Id}_\mathcal{A}$.

\begin{lem}\label{lem:peri}
Let $\mathcal{A}$ have a Serre duality which is perfect and periodic. Take any  periodicity isomorphism $\eta \colon \mathbb{S}^d\rightarrow {\rm Id}_\mathcal{A}$ for some $d\geq 1$. Then $\eta$ is compatible.
\end{lem}

\begin{proof}
By the assumption  $\sigma_\mathbb{S}={\rm Id}_{\mathbb{S}^2}$. By iterating Lemma \ref{lem:key}(1) we infer that $\sigma_{\mathbb{S}^d}$ equals the identity transformation on ${\mathbb{S}^{d+1}}$. We apply Lemma \ref{lem:key}(2) to $\eta \colon \mathbb{S}^d\rightarrow {\rm Id}_\mathcal{A}$ and deduce $\eta \mathbb{S}=\mathbb{S}\eta$, that is, the periodicity isomorphism $\eta$ is compatible.
\end{proof}

\section{The Serre duality on the category of equivariant objects}

In this section, we recall from \cite{De,RR, DGNO} some notation on the category of equivariant objects with respect to a finite group action. We compute explicitly its Serre duality. Special cases are exploited with examples. In particular, we study an abelian category with a periodic Serre duality and its equivariantization.

Let $G$ be a finite group, which is written multiplicatively and whose unit is denoted by $e$. Let $\mathcal{A}$ be a $k$-linear abelian category, which is Hom-finite and skeletally small.

\subsection{Finite group actions and equivariantizations}

We recall the notion of a group action on a category; see \cite{De,RR,DGNO}.   A \emph{$k$-linear action} of $G$ on $\mathcal{A}$ consists of the data $\{F_g, \varepsilon_{g, h}|\; g, h\in G\}$, where each $F_g\colon \mathcal{A}\rightarrow \mathcal{A}$ is a $k$-linear auto-equivalence and each $\varepsilon_{g, h}\colon F_gF_h\rightarrow F_{gh}$ is a natural isomorphism such that the following $2$-cocycle condition
\begin{align}\label{equ:2-coc}
\varepsilon_{gh, l}\circ \varepsilon_{g,h}F_l=\varepsilon_{g, hl}\circ F_g\varepsilon_{h,l}
\end{align}
holds for all $g, h, l\in G$.

The given $k$-linear action $\{F_g, \varepsilon_{g, h}|\; g, h\in G\}$ is \emph{strict} provided that each $F_g\colon \mathcal{A}\rightarrow \mathcal{A}$ is an automorphism and each isomorphism $\varepsilon_{g, h}$ is the identity. Therefore, a strict action coincides with  a group homomorphism from $G$ to the $k$-linear automorphism group of $\mathcal{A}$.

In what follows, we assume that there is a $k$-linear $G$-action  $\{F_g, \varepsilon_{g, h}|\; g, h\in G\}$ on $\mathcal{A}$. We observe that there exists a unique natural isomorphism $u\colon F_e\rightarrow {\rm Id}_\mathcal{A}$, called the \emph{unit} of the action, satisfying $\varepsilon_{e,e}=F_eu$.  Taking $h=e$ in (\ref{equ:2-coc}) we obtain that
\begin{align}\label{equ:2-coc2}
\varepsilon_{g,e}F_l=F_g\varepsilon_{e,l}.
 \end{align}
 Taking $g=e$ in (\ref{equ:2-coc2}) we infer that $uF_l=\varepsilon_{e,l}$; in particular, we have $uF_e=\varepsilon_{e,e}$. Taking $l=e$ in (\ref{equ:2-coc2}) and using the identity $\varepsilon_{e,e}=uF_e$, we infer that $\varepsilon_{g,e}=F_gu$.

Let $g\in G$. For each $d\geq 1$, we define a natural isomorphism $\varepsilon_g^d\colon F_g^d\rightarrow F_{g^d}$ as follows, where $F_g^d$ denotes the $d$-th power of $F_g$. We define $\varepsilon_g^1={\rm Id}_{F_g}$ and $\varepsilon_g^2=\varepsilon_{g, g}$. If $d>2$, we define $\varepsilon_g^d=\varepsilon_{g^{d-1}, g}\circ \varepsilon_g^{d-1}F_g$. It follows by (\ref{equ:2-coc}) and induction that \begin{align}\label{equ:2}
\varepsilon_g^{d}=\varepsilon_{g, g^{d-1}}\circ F_g\varepsilon_g^{d-1}.
\end{align}

We assume that $g^d=e$ for some $d\geq 1$. Consider the following isomorphisms
 $$\theta\colon F_g^d\stackrel{\varepsilon_{g}^d}\longrightarrow F_{g^d}=F_e\stackrel{u}\longrightarrow {\rm Id}_\mathcal{A}.$$
We claim that $\theta F_g=F_g\theta$. Indeed, by $uF_g=\varepsilon_{e, g}=\varepsilon_{g^d, g}$,  we have $\theta F_g=\varepsilon_g^{d+1}$. By (\ref{equ:2}) we have $\varepsilon_g^{d+1}=\varepsilon_{g, g^d}\circ F_g\varepsilon_g^{d}=\varepsilon_{g, e}\circ F_g\varepsilon_g^{d}$. Recall that $\varepsilon_{g,e}=F_gu$. It follows that $\varepsilon_{g}^{d+1}=F_g\theta$. Then we are done.

The above claim implies the following result.

\begin{lem}\label{lem:per}
Let $\{F_g, \varepsilon_{g, h}|\; g, h\in G\}$ be a $k$-linear action of $G$ on $\mathcal{A}$. Then each auto-equivalence $F_g$ has a compatible periodicity isomorphism.\hfill $\square$
\end{lem}

The following example is a partial converse of Lemma \ref{lem:per}.

\begin{exm}\label{exm:1}
{\rm Let $F\colon \mathcal{A}\rightarrow \mathcal{A}$ be a $k$-linear auto-equivalence with a compatible periodicity isomorphism $\theta\colon F^d\rightarrow {\rm Id}_\mathcal{A}$  of order $d$. We recall that the compatibility means  $F\theta=\theta F$.

Denote by $C_d=\langle g\; |\; g^d=1\rangle=\{e=g^0, g, \cdots, g^{d-1}\}$ the cyclic group of order $d$. Then $g^ig^j=g^{[i+j]}$ for $0\leq i, j\leq d-1$, where $[i+j]=i+j$ if $i+j\leq d-1$ and $[i+j]=i+j-d$ otherwise.

We now construct a $k$-linear $C_d$-action on $\mathcal{A}$ that is \emph{induced} by $F$ and $\theta$. For each $0\leq i\leq d-1$, we define $F_{g^i}=F^i$, where $F^0={\rm Id}_\mathcal{A}$. For $0\leq i, j\leq d-1$, we define the natural isomorphism $\varepsilon_{g^i, g^j}\colon F_{g^i}F_{g^j}\rightarrow F_{g^ig^j}=F_{g^{[i+j]}}$ as follows: $\varepsilon_{g^i, g^j}={\rm Id}_{F^{i+j}}$ if $i+j\leq d-1$, and $\varepsilon_{g^i, g^j}=F^{i+j-d}\theta$ otherwise.

We claim that the following $2$-cocycle condition
\begin{align}
\varepsilon_{g^{[i+j]}, g^l}\circ \varepsilon_{g^i,g^j}F_{g^l}=\varepsilon_{g^i, g^{[j+l]}}\circ F_{g^i}\varepsilon_{g^j,g^l}
\end{align}
holds. Then we are done with the construction.

Indeed, we have to verify it by the four cases depending on whether $i+j$ and $j+l$ are less than $d-1$ or not. Then for the two cases where $i+j\geq d$ we have to use the condition $\theta F^l=F^l\theta$ for any $0\leq l\leq d-1$.

We observe that the constructed $C_d$-action on $\mathcal{A}$ is strict if and only if $F$ is an automorphism such that $F^d={\rm Id}_\mathcal{A}$ and that $\theta$ is the identity transformation.}
\end{exm}

 Let $\{F_g, \varepsilon_{g, h}|\; g, h\in G\}$ be a $k$-linear $G$-action on $\mathcal{A}$. A \emph{$G$-equivariant object} in $\mathcal{A}$ is a pair $(X, \alpha)$, where $X$ is an object in $\mathcal{A}$ and $\alpha$ assigns for each $g\in G$ an isomorphism $\alpha_g\colon X\rightarrow F_gX$ subject to the relations
 \begin{align}\label{equ:rel}
 \alpha_{gg'}=(\varepsilon_{g,g'})_X \circ F_g(\alpha_{g'}) \circ \alpha_g.
 \end{align}
 These relations imply that $\alpha_e=u^{-1}_X$. A morphism $f\colon  (X, \alpha)\rightarrow (Y, \beta)$ between  two $G$-equivariant objects is a morphism $f\colon X\rightarrow Y$ in $\mathcal{A}$ such that
  $\beta_g\circ f=F_g(f)\circ \alpha_g$ for all $g\in G$. This gives rise to
   the category $\mathcal{A}^G$ of $G$-equivariant objects, and the \emph{forgetful functor}
   $U\colon \mathcal{A}^G\rightarrow \mathcal{A}$ defined by $U(X, \alpha)=X$ and $U(f)=f$. The process forming the category $\mathcal{A}^G$ of equivariant objects is known as the \emph{equivariantization} of $\mathcal{A}$ with respect to the group action; see \cite{DGNO}.

We observe that $\mathcal{A}^G$ is an abelian category; indeed, a sequence of $G$-equivariant objects is exact if and only if it is exact in $\mathcal{A}$. We further observe that $\mathcal{A}^G$ is naturally $k$-linear, which is Hom-finite.

\begin{exm}\label{exm:2}
{\rm Let $F\colon \mathcal{A}\rightarrow \mathcal{A}$ be a $k$-linear periodic auto-equivalence. We assume that there is a compatible periodicity isomorphism $\theta\colon F^d\rightarrow {\rm Id}_\mathcal{A}$; here, $d$ is a multiple of  the order of $F$. We consider the $C_d$-action on $\mathcal{A}$ induced by $F$ and $\theta$ in Example \ref{exm:1}. Then a $C_d$-equivariant object $(X, \alpha)$ is completely determined by the isomorphism $\alpha_g\colon X\rightarrow FX$, which satisfies $F^{d-1}(\alpha_g) \circ \cdots \circ F(\alpha_g)\circ \alpha_g=\theta^{-1}_X$; a morphism $f\colon (X, \alpha)\rightarrow (Y, \beta)$ of $C_d$-equivariant objects is a morphism $f\colon X\rightarrow Y$ satisfying $\beta_g\circ f=F(f)\circ \alpha_g$.

In case that $d$ equals the order of $F$, we denote the category $\mathcal{A}^{C_d}$ of $C_d$-equivariant objects by $\mathcal{A}\sslash F$. We now explain the reason for this notation.

Recall that the \emph{orbit category} $\mathcal{A}/F$ of $\mathcal{A}$ with respect to $F$ and $\theta$  is defined as follows; compare \cite[Subsection 3.1]{RR} and \cite{Kel}. The objects of $\mathcal{A}/F$ are the same with $\mathcal{A}$. For two objects $X, Y$ the Hom space is given by
$${\rm Hom}_{\mathcal{A}/F}(X, Y)=\bigoplus_{i=0}^{d-1} {\rm Hom}_\mathcal{A}(X, F^iY),$$
 whose elements are denoted by $(f_i)\colon X\rightarrow Y$ with $f_i\in {\rm Hom}_\mathcal{A}(X, F^iY)$ for each $0\leq i\leq d-1$. The composition of two morphisms $(f_i)\colon X\rightarrow Y$ and $(g_i)\colon Y\rightarrow Z$ is given by $(h_i)\colon X\rightarrow Z$, where $h_i=\sum_{\{0\leq j, l\leq d-1\; |\; [j+l]=i\}} \varepsilon_{g^j, g^l}\circ F^j(g_l)\circ f_j$. The following well-known fact is implicit in \cite[Lemma 4.4]{Chen}:  if $d$ is invertible in $k$, then $\mathcal{A}\sslash F$ is equivalent to the idempotent completion of the orbit category $\mathcal{A}/F$.

We mention that the categories $\mathcal{A}/F$ and $\mathcal{A}\sslash F$ depend on the choice of the compatible periodicity isomorphism $\theta$.}
\end{exm}

Let $\{F_g, \varepsilon_{g, h}|\; g, h\in G\}$ be a $k$-linear $G$-action on $\mathcal{A}$ as above. Let $(X, \alpha)$ and $(Y, \beta)$ be two objects in $\mathcal{A}^G$. Then the Hom space ${\rm Hom}_\mathcal{A}(X, Y)$ carries a $k$-linear $G$-action \emph{associated} to these two objects: for $g\in G$ and $f\colon X\rightarrow Y$, $g.f=\beta_g^{-1}\circ F_g(f)\circ \alpha_g$. There is a $k$-linear isomorphism induced by the forgetful functor $U$
\begin{align}\label{iso:1}
{\rm Hom}_{\mathcal{A}^G} ((X, \alpha), (Y, \beta))\stackrel{\sim}\longrightarrow {\rm Hom}_\mathcal{A}(X, Y)^G.
\end{align}
Here, for a  $k$-vector space $V$ with a $k$-linear $G$-action we denote by $V^G$ the invariant subspace.

Similarly, given two objects  $(X, \alpha)$ and $(Y, \beta)$ in $\mathcal{A}^G$, the Ext space ${\rm Ext}^1_\mathcal{A}(X, Y)$ carries a $k$-linear $G$-action \emph{associated} to these two objects: for $g\in G$ and $[\xi] \in {\rm Ext}^1_\mathcal{A}(X, Y)$, $g.[\xi]=\beta_g^{-1}.[F_g(\xi)].\alpha_g$. The forgetful functor $U$ induces a $k$-linear map
\begin{align}\label{map:1}
{\rm Ext}^1_{\mathcal{A}^G} ((X, \alpha), (Y, \beta))\longrightarrow {\rm Ext}^1_\mathcal{A}(X, Y)^G.
\end{align}
Indeed, for an extension $\xi\colon 0\rightarrow (Y, \beta)\rightarrow (E, \gamma)\rightarrow (X, \alpha)\rightarrow 0$ in $\mathcal{A}^G$ the corresponding extension $U(\xi)$ in $\mathcal{A}$  satisfies $\beta_g.[F_gU(\xi)]=[U(\xi)].\alpha_g$ by the following commutative diagram and Lemma \ref{lem:McL}(1). In other words,  $[U(\xi)]$ lies in the invariant subspace ${\rm Ext}^1_\mathcal{A}(X, Y)^G$.
\[\xymatrix{
U(\xi)\colon & 0\ar[r] & Y \ar[d]^-{\beta_g} \ar[r] & E\ar[r]\ar[d]^-{\gamma_g} & X\ar[r] \ar[d]^-{\alpha_g} & 0\\
F_gU(\xi)\colon & 0\ar[r] & F_gY \ar[r] & F_gE\ar[r] & F_gX\ar[r] & 0
}\]

The following fact is well known.

\begin{lem}\label{lem:inv}
Assume that the order $|G|$ of $G$ is invertible in $k$. Let $(X, \alpha)$ and $(Y, \beta)$ be as above. Then the map  (\ref{map:1}) is an isomorphism.
\end{lem}

 \begin{proof}
Recall that the Ext spaces in $\mathcal{A}$ are isomorphic to certain Hom spaces in the bounded derived category $\mathbf{D}^b(\mathcal{A})$. The $G$-action on $\mathcal{A}$  induces naturally a $G$-action on $\mathbf{D}^b(\mathcal{A})$.  We observe that the isomorphism (\ref{iso:1}) is valid for any $k$-linear category, in particular, it is valid for $\mathbf{D}^b(\mathcal{A})$. Then the isomorphism (\ref{map:1}) follows from \cite[Proposition 4.5]{Chen} and the corresponding isomorphism (\ref{iso:1}).
 \end{proof}

 \subsection{An explicit Serre duality for the equivariantization}

We assume that the abelian category $\mathcal{A}$ has the Serre duality (\ref{equ:S}). Consider the $k$-linear $G$-action $\{F_g, \varepsilon_{g, h}|\; g, h\in G\}$ as above.

 We apply Lemma \ref{lem:2} to each $k$-linear auto-equivalence $F_g$ and  obtain the commutator isomorphism $\sigma_{F_g}\colon F_g\mathbb{S}\rightarrow \mathbb{S}F_g$ for each $g\in G$, which will be denoted by $\sigma_g$. Let $(X, \alpha)$ be a $G$-equivariant object.  For each $g\in G$, we consider the following isomorphism
 $$\tilde{\alpha}_g=(\sigma_g)^{-1}_X\circ \mathbb{S}(\alpha_g)\colon \mathbb{S}X\longrightarrow F_g\mathbb{S}X.$$

\begin{lem}
Keep the notation as above. Then the isomorphisms $\tilde{\alpha}_g$'s satisfy the identity $\tilde{\alpha}_{gh}=(\varepsilon_{g,h})_{\mathbb{S}X} \circ F_g(\tilde{\alpha}_h)\circ \tilde{\alpha}_g$. In other words, the pair $(\mathbb{S}X, \tilde{\alpha})$ is a $G$-equivariant object in $\mathcal{A}$.
\end{lem}

\begin{proof}
We claim that the following identity holds
\begin{align}\label{equ:8}
(\varepsilon_{g, h})_{\mathbb{S}X}\circ F_g((\sigma_h)^{-1}_X)\circ (\sigma_g)^{-1}_{F_hX}=(\sigma_{gh})_X^{-1}\circ \mathbb{S}((\varepsilon_{g,h})_X).\end{align}
Indeed, by Lemma \ref{lem:key}(1) we have $(\sigma_g)_{F_hX}\circ F_g((\sigma_h)_X)=(\sigma_{F_gF_h})_X$. Applying Lemma \ref{lem:key}(2) to $\varepsilon_{g,h}$, we have $(\sigma_{gh})_X\circ (\varepsilon_{g,h})_{\mathbb{S}X}=\mathbb{S}((\varepsilon_{g,h})_X)\circ (\sigma_{F_gF_h})_X$. Putting these two identities together, we infer the claim.

We complete the proof by the following identity:
\begin{align*}
(\varepsilon_{g, h})_{\mathbb{S}X}\circ F_g(\tilde{\alpha}_h)\circ \tilde{\alpha}_g &=(\varepsilon_{g, h})_{\mathbb{S}X}\circ F_g((\sigma_h)^{-1}_X)\circ F_g\mathbb{S}(\alpha_h)\circ (\sigma_g)^{-1}_X\circ \mathbb{S}(\alpha_g)\\
&=(\varepsilon_{g, h})_{\mathbb{S}X}\circ F_g((\sigma_h)^{-1}_X)\circ (\sigma_g)^{-1}_{F_hX}\circ \mathbb{S}F_g(\alpha_h)\circ \mathbb{S}(\alpha_g)\\
&=(\sigma_{gh})_X^{-1}\circ \mathbb{S}((\varepsilon_{g,h})_X) \circ \mathbb{S}F_g(\alpha_h)\circ \mathbb{S}(\alpha_g)\\
&=(\sigma_{gh})_X^{-1} \circ \mathbb{S}((\varepsilon_{g,h})_X \circ F_g(\alpha_h)\circ \alpha_g)\\
&=(\sigma_{gh})_X^{-1} \circ \mathbb{S}(\alpha_{gh})=\tilde{\alpha}_{gh}.
\end{align*}
Here, the second equality uses the naturalness of $\sigma_g$, and the third uses the claim above.
\end{proof}

We define a functor $\mathbb{S}^G\colon \mathcal{A}^G\rightarrow \mathcal{A}^G$ by sending $(X, \alpha)$ to $(\mathbb{S}X, \tilde{\alpha})$, a morphism $f\colon (X, \alpha)\rightarrow (Y, \beta)$ to $\mathbb{S}(f)\colon (\mathbb{S}X, \tilde{\alpha})\rightarrow (\mathbb{S}Y, \tilde{\beta})$. Since $\mathbb{S}$ is a $k$-linear auto-equivalence on $\mathcal{A}$, it follows that $\mathbb{S}^G$ is a $k$-linear auto-equivalence on $\mathcal{A}^G$.

Let $(X, \alpha)$ and $(Y, \beta)$ be two objects in $\mathcal{A}^G$. Consider the following map
$$\psi\colon {\rm Hom}_{\mathcal{A}^G}((Y, \beta), \mathbb{S}^G(X, \alpha))\longrightarrow D{\rm Ext}^1_{\mathcal{A}^G}((X, \alpha), (Y, \beta))$$
defined by $\psi(f)([\xi])={\rm Tr}_X(U(f.[\xi]))$, where $f\colon (Y, \beta)\rightarrow \mathbb{S}^G(X, \alpha)$ and $[\xi]\in {\rm Ext}^1_{\mathcal{A}^G}((X, \alpha), (Y, \beta))$, and $U\colon \mathcal{A}^G\rightarrow \mathcal{A}$ is the forgetful functor.  We observe that $\psi(f)([\xi])=\phi_{X, Y}^{-1}(U(f))([U(\xi)])$. It follows that $\psi$ is natural in both $(X, \alpha)$ and $(Y, \beta)$.

The main result describes explicitly the Serre duality on $\mathcal{A}^G$; it is somehow analogous to \cite[Lemma 2.1 a)]{Kel} for orbit categories.

\begin{thm}\label{thm:1}
Let $G$ be a finite group acting on $\mathcal{A}$. Assume that the order $|G|$ of $G$ is invertible in $k$. Then the above map $\psi$ is an isomorphism. Consequently, the category $\mathcal{A}^G$ has a Serre duality given by $\psi^{-1}$
 \begin{align}\label{equ:S-equi}
 D{\rm Ext}^1_{\mathcal{A}^G}((X, \alpha), (Y, \beta))\stackrel{\sim}\longrightarrow {\rm Hom}_{\mathcal{A}^G}((Y, \beta), \mathbb{S}^G(X, \alpha)) \end{align}
where $\mathbb{S}^G\colon \mathcal{A}^G\rightarrow \mathcal{A}^G$ is the Serre functor.
\end{thm}

\begin{proof}
We recall that associated to the objects $(Y, \beta)$ and $\mathbb{S}^G(X, \alpha)$, the $G$-action on  ${\rm Hom}_\mathcal{A}(Y, \mathbb{S}X)$ is given by $g.f=\tilde{\alpha}^{-1}_g\circ F_g(f)\circ \beta_g=\mathbb{S}(\alpha_g^{-1})\circ (\sigma_g)_X\circ F_g(f)\circ \beta_g$. On the other hand, associated to $(X, \alpha)$ and $(Y, \beta)$,  the $G$-action on  ${\rm Ext}_\mathcal{A}^1(X, Y)$ is given by $g.[\xi]=\beta_g^{-1}.[\xi].\alpha_g$. Then $G$ acts on $D{\rm Ext}_\mathcal{A}^1(X, Y)$ by the contragredient action.

We claim that the isomorphism
$$\phi^{-1}=\phi_{X, Y}^{-1}\colon {\rm Hom}_\mathcal{A}(Y, \mathbb{S}X)\longrightarrow D{\rm Ext}_\mathcal{A}^1(X, Y)$$ is compatible with these $G$-actions. Recall that $\phi^{-1}(f)([\xi])={\rm Tr}_X(f.[\xi])$ for each $[\xi]\in {\rm Ext}_\mathcal{A}^1(X, Y)$. Take $g\in G$. We assume that
\begin{align}\label{equ:7}
\beta_g.[\xi].\alpha_g^{-1}=F_g([\xi'])
 \end{align}
 for some $[\xi']\in {\rm Ext}_\mathcal{A}^1(X, Y)$. It follows that $[\xi]=g.[\xi']$ and thus $[\xi']=g^{-1}.[\xi]$. The claim amounts to the equation $\phi^{-1}(g.f)=g.\phi^{-1}(f)$.

The following identity proves the claim:
\begin{align*}
\phi^{-1}(g.f)([\xi]) &={\rm Tr}_X((g.f).[\xi])\\
                    &  ={\rm Tr}_X((\mathbb{S}(\alpha_g^{-1})\circ (\sigma_g)_X\circ F_g(f)\circ \beta_g).[\xi])\\
                    &={\rm Tr}_{F_gX}(((\sigma_g)_X\circ F_g(f)\circ \beta_g).[\xi].\alpha_g^{-1})\\
                     &={\rm Tr}_{F_gX}((\sigma_g)_X. F_g(f.[\xi']))\\
                     & ={\rm Tr}_X(f.[\xi'])\\
                     &=\phi^{-1}(f)(g^{-1}.[\xi])=(g.\phi^{-1}(f))([\xi]).
\end{align*}
Here, for the third equality we use Lemma \ref{lem:1}(2), for the fourth we use (\ref{equ:7}), and for the fifth we apply (\ref{equ:1}) to $F_g$. The last equality follows from the definition of the contragredient action.

It follows from the claim that the isomorphism $\phi^{-1}$ induces isomorphisms $${\rm Hom}_\mathcal{A}(Y, \mathbb{S}X)^G \stackrel{\sim}\longrightarrow (D{\rm Ext}_\mathcal{A}^1(X, Y))^G\simeq D({\rm Ext}_\mathcal{A}^1(X, Y)^G)$$ of
$G$-invariants. Then we are done by (\ref{iso:1}) and Lemma \ref{lem:inv}. We observe that the above isomorphism is identified with the map $\psi$. Then we are done.
\end{proof}

\begin{rem}\label{rem:2}
The trace map corresponding to (\ref{equ:S-equi}) is given by
$${\rm Tr}_{(X, \alpha)}\colon {\rm Ext}^1_{\mathcal{A}^G}((X, \alpha), \mathbb{S}^G(X, \alpha))\longrightarrow k$$
such that ${\rm Tr}_{(X, \alpha)}([\xi])={\rm Tr}_X([U(\xi)])$. In view of (\ref{equ:4}), we have the following immediate observation: if the Serre duality on $\mathcal{A}$ is perfect, so is the one on $\mathcal{A}^G$.
\end{rem}

\subsection{Special case I: the trivial Serre functor}

Let $\{F_g, \varepsilon_{g, h}|\; g, h\in G\}$ a $k$-linear $G$-action on $\mathcal{A}$. We will work in the following setup.

\vskip 5pt

\noindent{\bf Setup A:}\quad There is a central element $g\in G$ such that $F_g=\mathbb{S}$, the Serre functor on $\mathcal{A}$. Moreover, for each $h\in G$ we have $\varepsilon_{g, h}^{-1}\circ\varepsilon_{h, g}=\sigma_{h}$.

 \vskip 5pt
  Here, we recall that $\sigma_h=\sigma_{F_h}\colon F_h\mathbb{S}\rightarrow \mathbb{S}F_h$ denotes the commutator isomorphism associated to $F_h$.

 \begin{prop}\label{prop:A}
 Assume that we are in Setup A. Then the auto-equivalence  $\mathbb{S}^G$ on $\mathcal{A}^G$ is isomorphic to the identity functor. In particular, if further the order $|G|$ of $G$ is invertible in $k$, then the Serre functor on  $\mathcal{A}^G$ is isomorphic to the identity functor.
 \end{prop}

 \begin{proof}
We  construct an explicit isomorphism $\delta\colon {\rm Id}_{\mathcal{A}^G}\rightarrow \mathbb{S}^G$. For an object $(X, \alpha)$, we define $\delta_{(X, \alpha)}=\alpha_g\colon (X, \alpha)\rightarrow (\mathbb{S}X, \tilde{\alpha})$. We claim that $\delta_{(X, \alpha)}$ is a morphism of equivariant objects, that is, for each $h\in G$ we have $\tilde{\alpha}_h\circ \alpha_g=F_h(\alpha_g)\circ \alpha_h$.

 Indeed, we have the following identity
 \begin{align*}
 \tilde{\alpha}_h\circ \alpha_g&= \sigma_h^{-1}\circ \mathbb{S}(\alpha_h)\circ \alpha_g\\
                               &= \varepsilon_{h, g}^{-1}\circ \varepsilon_{g, h}\circ \mathbb{S}(\alpha_h)\circ \alpha_g\\
                               &=\varepsilon_{h, g}^{-1}\circ \alpha_{gh}\\
                               &=\varepsilon_{h, g}^{-1}\circ \alpha_{hg}\\
                               &=F_h(\alpha_g)\circ \alpha_h,
                              \end{align*}
 where the second and fourth equalities use the assumptions in Setup A, and the third and final ones use (\ref{equ:rel}).

 The naturalness of $\delta$ in $(X, \alpha)$ is direct to verify. The last statement follows from Theorem \ref{thm:1}.
 \end{proof}

\begin{rem}\label{rem:3}
We replace $\mathbb{S}^G$ by ${\rm Id}_{\mathcal{A}^G}$ via the isomorphism $\delta$. Then the new trace map on $\mathcal{A}^G$ is given by
$${\rm Tr}'_{(X, \alpha)}\colon {\rm Ext}^1_{\mathcal{A}^G}((X, \alpha), (X, \alpha))\longrightarrow k$$
such that ${\rm Tr}'_{(X, \alpha)}([\xi])={\rm Tr}_X(\alpha_g.[U(\xi)])$; compare Remarks \ref{rem:1} and \ref{rem:2}.
\end{rem}

We consider a special case of Proposition \ref{prop:A}. It justifies a statement in \cite{Lentalk}, that is, factoring out the Serre functor yields an abelian category with a trivial Serre functor; compare \cite[Subsection 2.4]{Kel}.

We assume that $\mathcal{A}$ has a Serre duality which is perfect and periodic of order $d$. Here, $d$ equals the order of $\mathbb{S}$. By Lemma \ref{lem:peri} there is a compatible periodicity isomorphism $\eta\colon \mathbb{S}^d\rightarrow {\rm Id}_\mathcal{A}$. We consider the $C_d$-action on $\mathcal{A}$ induced by $\mathbb{S}$ and $\eta$ in Example \ref{exm:1}. Recall the notation $\mathcal{A}\sslash \mathbb{S}=\mathcal{A}^{C_d}$ from Example \ref{exm:2}.

\begin{cor}\label{cor:A}
Let  $\mathcal{A}$ have a Serre duality which is perfect and periodic of order $d$. Assume that $d$ is invertible in $k$. Then the Serre functor on the category $\mathcal{A}\sslash \mathbb{S}$ is isomorphic to the identity functor.
\end{cor}

\begin{proof}
By the construction in Example \ref{exm:1}, the conditions in Setup A are trivially satisfied. Then the result follows from Proposition \ref{prop:A}.
\end{proof}

\subsection{Special case II: a peroidic Serre duality}
Let $\{F_g, \varepsilon_{g, h}|\; g, h\in G\}$  be a $k$-linear $G$-action on $\mathcal{A}$. We will work in the following setup.

\vskip 5pt

\noindent{\bf Setup B:}\quad Each auto-equivalence $F_g\colon \mathcal{A}\rightarrow \mathcal{A}$ is $Z(\mathcal{A})$-linear. In particular, we obtain a group homomorphism $G\rightarrow {\rm Aut}(\mathcal{A})$ sending $g$ to $[F_g]$.

\vskip 5pt

We recall a general fact; compare \cite[Subsection 4.1.3]{DGNO}.

\begin{fact}\label{fact:1}
Assume that we are in Setup B. Let $\rho\colon G\rightarrow Z(\mathcal{A})^\times$ be an arbitrary group homomorphism. For an object $(X, \alpha)$ in $\mathcal{A}^G$, we define another object $(X, \rho\otimes\alpha)$ such that $(\rho\otimes\alpha)_g=\rho(g^{-1}).\alpha_g$ for each $g\in G$; here, we use ``." to denote the $Z(\mathcal{A})$-action on the Hom spaces in $\mathcal{A}$. Observe that it is indeed a $G$-equivariant object. This gives rise to an automorphism on $\mathcal{A}^G$
$$\rho\otimes-\colon \mathcal{A}^G\longrightarrow \mathcal{A}^G,$$
which sends $(X, \alpha)$ to $(X, \rho\otimes \alpha)$ and acts on morphisms by the identity.
\end{fact}

Let $\mathcal{A}$ have a periodic Serre duality. We take a periodicity isomorphism $\eta\colon \mathbb{S}^d\rightarrow {\rm Id}_\mathcal{A}$ for some $d\geq 1$. We recall from Subsection 2.3 the induced homomorphism $\kappa\colon {\rm Aut}(\mathcal{A})\rightarrow Z(\mathcal{A})^\times$. For each $g\in G$, we write $\kappa(g)=\kappa([F_g])$. This defines a group homomorphism
$$\kappa\colon G\longrightarrow Z(\mathcal{A})^\times,$$
which is referred as the \emph{induced homomorphism} associated to the group action and the periodicity isomorphism $\eta$. In particular, by Fact \ref{fact:1} we have the automorphism $\kappa \otimes-\colon \mathcal{A}^G\rightarrow \mathcal{A}^G$.

\begin{prop}\label{prop:B}
Let $\mathcal{A}$ have a periodic Serre duality with a periodicity isomorphism $\eta\colon \mathbb{S}^d\rightarrow {\rm Id}_\mathcal{A}$. Assume that we are in Setup B. Then we have an isomorphism of endofunctors on $\mathcal{A}^G$
$$(\mathbb{S}^G)^d\stackrel{\sim}\longrightarrow \kappa\otimes- .$$
In particular, the auto-equivalence $\mathbb{S}^G$ on $\mathcal{A}^G$ is periodic.
\end{prop}

\begin{proof}
Let $(X, \alpha)$ be an arbitrary object in $\mathcal{A}^G$. We observe that $(\mathbb{S}^G)^d(X, \alpha)=(\mathbb{S}^dX, \tilde{\alpha}^d)$, where $\tilde{\alpha}^d_g=(\sigma_g^d)^{-1}_X \circ \mathbb{S}^d(\alpha_g)$ for each $g\in G$. Here,  $\sigma_g^d=\sigma^d_{F_g}\colon F_g\mathbb{S}^d\rightarrow \mathbb{S}^dF_g$ is the $d$-th commutator isomorphism associated to $F_g$. On the other hand, $(\kappa \otimes-)(X, \alpha)=(X, \kappa\otimes \alpha)$, where $(\kappa\otimes \alpha)_g=\kappa(F_g)^{-1}.\alpha_g$. We claim that $$\eta_X\colon (\mathbb{S}^dX, \tilde{\alpha}^d)\longrightarrow (X, \kappa\otimes\alpha)$$
 is a morphism, and thus an isomorphism, in $\mathcal{A}^G$.

 It suffices to show that $F_g(\eta_X)\circ \tilde{\alpha}_g^d=(\kappa(F_g)^{-1}.\alpha_g)\circ \eta_X$. By (\ref{equ:3}) we have $F_g(\eta_X)\circ (\sigma_{F_g}^d)^{-1}_X=(t_{F_g})^{-1}\circ \eta_{F_gX}=\kappa(F_g)^{-1}.\eta_{F_gX}$. It follows that  $F_g(\eta_X)\circ \tilde{\alpha}_g^d=(\kappa(F_g)^{-1}.\eta_{F_gX})\circ \mathbb{S}^d(\alpha_g)$, which equals, by the naturalness of $\eta$, $(\kappa(F_g)^{-1}.\alpha_g)\circ \eta_X$. We are done with the claim.

 We observe that the above isomorphism $\eta_X$ is natural in $(X, \alpha)$. This proves the required isomorphism of functors. Since the $|G|$-th power of the automorphism $\kappa\otimes-$ equals the identity functor, the isomorphism implies that $\mathbb{S}^G$ is periodic.
\end{proof}

We observe the following immediate consequence of Theorem \ref{thm:1} and Proposition \ref{prop:B}.

\begin{cor}\label{cor:B}
Let $\mathcal{A}$ have a  periodic Serre duality.  Assume that we are in Setup B and that the order $|G|$ of $G$ is invertible in $k$. Then the category $\mathcal{A}^G$ has a  periodic Serre duality.\hfill $\square$
\end{cor}

\subsection{Special case III: a more explicit Serre functor}

Let $\{F_g, \varepsilon_{g, h}|\; g, h\in G\}$  be a $k$-linear $G$-action on $\mathcal{A}$. We will work in the following setup.

\vskip 5pt

\noindent {\bf Setup C.} \quad Each auto-equivalence $F_g\colon \mathcal{A}\rightarrow \mathcal{A}$ is $Z(\mathcal{A})$-linear satisfying $F_g\mathbb{S}=\mathbb{S}F_g$; moreover, we assume that $\mathbb{S}\varepsilon_{g,h}=\varepsilon_{g, h}\mathbb{S}$.

\vskip 5pt

We consider the commutator isomorphism $\sigma_g=\sigma_{F_g}\colon F_g\mathbb{S}\rightarrow \mathbb{S}F_g=F_g\mathbb{S}$. By Lemma \ref{lem:Z} there is a unique element $\gamma(g)\in Z(\mathcal{A})^\times$ such that
\begin{align}\label{equ:5}
\sigma_g=\gamma(g) \; {\rm Id}_{F_g\mathbb{S}}.
\end{align}
By (\ref{equ:8}) and the assumption $\mathbb{S}\varepsilon_{g,h}=\varepsilon_{g, h}\mathbb{S}$,  we infer $\gamma(gg')=\gamma(g)\gamma(g')$ for any $g, g'\in G$. In other words, we have a group homomorphism
$$\gamma\colon G\longrightarrow Z(\mathcal{A})^\times,$$
which is referred as the \emph{commutator homomorphism} associated to the group action. In particular, by Fact \ref{fact:1} we have the automorphism $\gamma\otimes-\colon \mathcal{A}^G\rightarrow \mathcal{A}^G.$

We observe that for an object $(X, \alpha)$ in $\mathcal{A}^G$, the pair $(\mathbb{S}X, \mathbb{S}(\alpha))$ is also a $G$-equivariant object, where $\mathbb{S}(\alpha)_g=\mathbb{S}(\alpha_g)\colon \mathbb{S}X\rightarrow F_g\mathbb{S}X=\mathbb{S}F_gX$ for each $g\in G$. In this way, we obtain an auto-equivalence $ \mathcal{A}^G\rightarrow \mathcal{A}^G$ sending $(X, \alpha)$ to $(\mathbb{S}X, \mathbb{S}(\alpha))$, and a morphism $f$ to $\mathbb{S}(f)$; by abuse of notation, this auto-equivalence is denoted by $\mathbb{S}\colon \mathcal{A}^G\rightarrow \mathcal{A}^G$.

The following result describes more explicitly the auto-equivalence $\mathbb{S}^G$ on  $\mathcal{A}^G$.

\begin{prop}\label{prop:C}
Assume that we are in Setup C. Keep the notation as above. Then as endofunctors on $\mathcal{A}^G$ we have
$$\mathbb{S}^G=(\gamma\otimes-) \; \mathbb{S}.$$
\end{prop}

\begin{proof}
Let $(X, \alpha)$ be an arbitrary object in $\mathcal{A}^G$. Recall that $\mathbb{S}^G(X, \alpha)=(\mathbb{S}X, \tilde{\alpha})$, where in view of (\ref{equ:5}) we have $\tilde{\alpha}_g=\gamma(g^{-1}).\mathbb{S}(\alpha_g)$ for each $g\in G$. In other words, we have $\tilde{\alpha}=\gamma\otimes \mathbb{S}(\alpha)$ and thus $\mathbb{S}^G(X, \alpha)=((\gamma\otimes-)\; \mathbb{S})(X, \alpha)$. This proves that the two endofunctors agree on objects. It is obvious that they agree on morphisms. Then we are done.
\end{proof}

\subsection{Examples} We will close the note with our motivating examples; see \cite{GL87,Lentalk,CCZ}. Let $k$ be a field whose characteristic is not $2$,  and let $\lambda\in k$ which is not $0$ or $1$. Denote by $C_2=\{e, g\}$ the cyclic group of order $2$.

We denote by $\mathbb{X}$ the weighted projective line in the sense of \cite{GL87} with weight sequence $(2, 2,2,2)$ and parameter sequence $(\infty, 0, 1, \lambda)$. We denote by $\mathbb{E}$ the projective plane curve defined by the equation $y^2z=x(x-z)(x-\lambda z)$, which is a smooth elliptic curve. Recall that the category $\mbox{coh-}\mathbb{E}$ of coherent sheaves on $\mathbb{E}$ has a Serre duality such that its Serre functor being the identity functor. In particular, the category $\mbox{coh-}\mathbb{E}$ has a perfect Serre duality; see Example \ref{exm:3}.

The category $\mbox{coh-}\mathbb{X}$ of coherent sheaves on $\mathbb{X}$ has a Serre duality with the Serre functor $\mathbb{S}$ being the degree-shifting functor given by the dualizing element; in particular, we have  $\mathbb{S}^2={\rm Id}_{\mbox{coh-}\mathbb{X}}$. Moreover, the category $\mbox{coh-}\mathbb{X}$  has a perfect Serre duality. Here, the perfectness is implicitly contained in the proof of the Serre duality in \cite[Subsection 2.2]{GL87}. On the other hand, one might deduce the perfectness from the equivalence (\ref{equ:6}) and Remark \ref{rem:2}.

  We have the strict $C_2$-action on $\mbox{coh-}\mathbb{X}$ induced by $\mathbb{S}$ and the identity transformation; see Example \ref{exm:1}. It follows from Corollary \ref{cor:A} that $\mbox{coh-}\mathbb{X}\sslash \mathbb{S}=(\mbox{coh-}\mathbb{X})^{C_2}$ has a Serre duality such that its Serre functor is isomorphic to the identity functor. However, this is known since we have an equivalence $\mbox{coh-}\mathbb{X}\sslash \mathbb{S}\stackrel{\sim}\longrightarrow \mbox{coh-}\mathbb{E}$; see \cite{Lentalk} and \cite[Theorem 7.7]{CCZ} for details.

We observe an automorphism $\sigma$ of order $2$ on $\mathbb{E}$ such that $\sigma(x)=x$, $\sigma(y)=-y$ and $\sigma(z)=z$.  This gives rise to a strict $C_2$-action such that $F_g=\sigma_*$, the direct image functor, on  $\mbox{coh-}\mathbb{E}$. Then the conditions in Setup C are satisfied. Here, we recall that $Z(\mbox{coh-}\mathbb{E})$ is isomorphic to $k$. Denote by $\gamma\colon C_2\rightarrow k^\times$ the commutator homomorphism. It follows from Proposition \ref{prop:C} that the Serre functor on $(\mbox{coh-}\mathbb{E})^{C_2}$ equals $\gamma\otimes -$. However, we recall from \cite[Example 5.8]{GL87} and \cite[Theorem 7.7]{CCZ} the equivalence
\begin{align}\label{equ:6}
(\mbox{coh-}\mathbb{E})^{C_2}\stackrel{\sim}\longrightarrow  \mbox{coh-}\mathbb{X}.
 \end{align}
  In particular, the Serre functor on $(\mbox{coh-}\mathbb{E})^{C_2}$ is not isomorphic to the identity functor. This forces that the commutator homomorphism $\gamma$ is non-trivial, that is, $\gamma(g)=-1$. Thus we have described explicitly the Serre functor on $(\mbox{coh-}\mathbb{E})^{C_2}$.

Indeed, it would be very nice to compute directly the commutator homomorphism $\gamma\colon C_2\rightarrow k^\times$. In general, we ask the following question.

Let $\mathbb{Y}$ be a projective elliptic curve which is embedded in a projective $n$-space $\mathbb{P}^n$. Assume that $G$ is a finite subgroup of $GL_{n+1}(k)$ which acts on $\mathbb{P}^n$ naturally with $G.\mathbb{Y}=\mathbb{Y}$. In this way, we have a strict $k$-linear $G$-action on $\mbox{coh-}\mathbb{Y}$ given by the direct image functors. This action satisfies the conditions in Setup C, where the Serre functor on  $\mbox{coh-}\mathbb{Y}$ is taken to be the identity functor. What is the commutator homomorphism $\gamma \colon G\rightarrow k^\times$ for this action? It seems that $\gamma$ might coincide with the determinant map of the matrices in $G$.

\vskip 15pt

\noindent {\bf Acknowledgements}\; The author is grateful to Helmut Lenzing for giving him the reference \cite{Lentalk}. The author  thanks Zengqiang Lin for drawing his attention to the reference \cite{Kel}, and thanks Jianmin Chen and Dirk Kussin for helpful comments. This work is supported by NCET-12-0507, National Natural Science Foundation of China (No. 11201446) and the Alexander von Humboldt Stiftung.

\bibliography{}

\vskip 15pt

{\tiny \noindent   Xiao-Wu Chen \\
 School of Mathematical Sciences,\\
  University of Science and Technology of China, Hefei 230026, Anhui, PR China \\
  Wu Wen-Tsun Key Laboratory of Mathematics,\\
 USTC, Chinese Academy of Sciences, Hefei 230026, Anhui, PR China.\\
URL: http://home.ustc.edu.cn/$^\sim$xwchen.}

\end{document}